\documentclass[12pt,a4paper]{amsart}

\usepackage{graphics,epic}
\usepackage{amsmath,amssymb, amsthm}
\usepackage[all,2cell]{xy}

\newtheorem{theorem}{Theorem}[section]
\newtheorem*{theorem*}{Theorem}

\newtheorem{lemma}[theorem]{Lemma}
\newtheorem{proposition}[theorem]{Proposition}
\newtheorem{corollary}[theorem]{Corollary}
\newtheorem{conjecture}[theorem]{Conjecture}
\newtheorem*{conjecture*}{Conjecture}

\newtheorem{remark}[theorem]{Remark}

\newcommand{\wt}[1]{\widetilde{#1}}

\newcommand{\opname}[1]{\operatorname{\mathsf{#1}}}

\renewcommand{\mod}{\opname{mod}\nolimits}

\newcommand{\Mod}{\opname{Mod}\nolimits}

\newcommand{\per}{\opname{per}\nolimits}
\newcommand{\add}{\opname{add}\nolimits}

\newcommand{\der}{\cd}

\newcommand{\dimv}{\underline{\dim}\,}

\newcommand{\Z}{\mathbb{Z}}
\newcommand{\N}{\mathbb{N}}

\renewcommand{\P}{\mathbb{P}}

\newcommand{\Hom}{\opname{Hom}}
\newcommand{\go}{\opname{G_0}}
\newcommand{\RHom}{\opname{RHom}}

\newcommand{\Ext}{\opname{Ext}}

\newcommand{\End}{\opname{End}}

\newcommand{\ten}{\otimes}
\newcommand{\lten}{\overset{\boldmath{L}}{\ten}}

\newcommand{\ca}{{\mathcal A}}

\newcommand{\cc}{{\mathcal C}}
\newcommand{\cd}{{\mathcal D}}

\newcommand{\cf}{{\mathcal F}}

\newcommand{\ch}{{\mathcal H}}

\newcommand{\ct}{{\mathcal T}}
\newcommand{\cu}{{\mathcal U}}
\newcommand{\cv}{{\mathcal V}}

\renewcommand{\hat}[1]{\widehat{#1}}

\begin{document}

\title[c-vectors via $\tau$-tilting]{C-vectors via $\tau$-tilting theory}

\author{Changjian Fu}
\address{Changjian Fu\\Department of Mathematics\\SiChuan University\\610064 Chengdu\\P.R.China}
\email{changjianfu@scu.edu.cn}
\begin{abstract}
Inspired by the tropical dualities in cluster algebras, we introduce $c$-vectors for finite-dimensional algebras via $\tau$-tilting theory. Let $A$ be a finite-dimensional algebra over a field $k$.  Each  $c$-vector of $A$ can be realized as the (negative) dimension vector of certain indecomposable $A$-module and hence we establish the sign-coherence property for this kind of $c$-vectors. We then study the positive $c$-vectors for certain classes of finite-dimensional algebras including quasitilted algebras and cluster-tilted algebras. In particular, we recover the equalities of $c$-vectors for acyclic cluster algebras and skew-symmetric cluster algebras of finite type respectively obtained by N\'{a}jera Ch\'{a}vez.
 To this end,  a short proof for the sign-coherence of $c$-vectors for skew-symmetric cluster algebras has been given in the appendix.
\end{abstract}

\maketitle
\section{introduction}
The $c$-vectors and $g$-vectors introduced by Fomin-Zelevinsky~\cite{FZ07} are two kinds of integer vectors, which have played important roles in the theory of cluster algebras with coefficients. Both the vectors
are conjectured to have a so-called sign-coherence property ~\cite{FZ07}, which has been recently proved by Gross-Hacking-Keel-Kontsevich~\cite{GHKK14}.    For skew-symmetric cluster algebras, Nakanishi~\cite{N11} found the so-called tropical dualities between $c$-vectors and $g$-vectors  ({\it cf.} also~\cite{Keller12, Nagao13, Plamondon11} ).
With the assumption of sign-coherence of $c$-vectors, the tropical dualities between $c$-vectors and $g$-vectors has been further generalized to skew-symmetrizable cluster algebras by Nakanishi-Zelevinsky~\cite{NZ12}.
Moreover, they showed that many properties or conjectures of cluster algebras follow from the tropical dualities and hence follow from the sign-coherence of $c$-vectors.

  On the other hand, $c$-vectors may be seen as a generalization of root systems. It follows from Nagao's work~\cite{Nagao13} that  each $c$-vector of a given skew-symmetric cluster algebra can be realized as the (negative) dimension  vector of certain exceptional  module for the associated Jacobian algebra.  In particular, the set of $c$-vectors of an acyclic cluster algebra is a subset of the real  Schur roots for the corresponding Kac-Moody algebra. In ~\cite{Chavez13},  N\'{a}jera Ch\'{a}vez  showed the inverse inclusion is also true for acyclic cluster algebras ({\it cf.} also ~\cite{ST12}). Moreover, he also proved in ~\cite{Chavez14} that the set of  positive $c$-vectors of a skew-symmetric cluster algebra of finite type coincides with the set of dimension vectors of all the exceptional modules over the corresponding representation-finite cluster-tilted algebra.  Nakanishi-Stella ~\cite{NS14} gave a diagrammatic description of $c$-vectors for cluster algebras of finite type.  They proposed the root conjecture for any cluster algebras: for any skew-symmetrizable matrix $B$ any $c$-vector  of the cluster algebra $\ca(B)$ is a root of the associated Kac-Moody algebra $\mathfrak{g}(A(B))$ , where $A(B)$ is the Cartan counterpart of $B$. We refer to ~\cite{NS14} for more details on $c$-vectors of cluster algebras of finite type.

  In this paper, we pursue the representation-theoretic approach to study $c$-vectors.  We introduce the notion of $c$-vector for any finite-dimensional algebras via $\tau$-tilting theory ~\cite{AIR12} and study its relation with dimension vectors of indecomposable $\tau$-rigid and exceptional modules.  Let $A$ be a finite-dimensional algebra over a field $k$. In ~\cite{AIR12}, the authors showed that the indices $\opname{ind}(M,Q)$  of a basic support $\tau$-tilting pair $(M,Q)$ form a $\Z$-basis for the Grothendieck groups $\go(\per A)$ of the perfect derived category $\per A$. Let $\der^{b}(\mod A)$ be the bounded derived category of finitely generated right $A$-modules and $\langle-,-\rangle_A:\go(\per A)\times \der^{b}(\mod A)\to k$ the non-degenerate Euler bilinear form.
 We then define the $c$-vectors associated to $(M,Q)$ to be the dual basis of $\opname{ind}(M,Q)$ in $\go(\der^{b}(\mod A))$ with respect to  the Euler bilinear form $\langle-,-\rangle_A$.   Equivalently, one can define $c$-vectors as the dimension vectors of  $2$-term simple-minded collections~\cite{BY13} in $\der^b(\mod A)$ with respect to the canonical basis of simple $A$-modules.
   Using the bijection between $2$-term silting objects in $\per A$ and the intermediate $t$-structure on $\der^{b}(\mod A)$  in~\cite{KY12,BY13}, we show the sign-coherence property  holds for this kind of $c$-vectors.  If $A$ is a Jacobian-finite algebra, it follows from ~\cite{FuKeller10, BY13} and the tropical dualities that the $c$-vectors we obtained here do coincide with the one for the corresponding cluster algebra.
On the other hand, if $A$ is the preprojective algebra of Dynkin type $Q$, it follows from ~\cite{Mizuno14} that the set of $c$-vectors of $A$ coincides with the root system of the simple Lie algebra of  type $Q$.
Both the results suggest that it is worth investigating the $c$-vectors for a general finite-dimensional algebra and its possible connections with Lie theory. In the present paper, we mainly study the $c$-vectors for quasitilted algebras and show how to use the idea to recover the equalities of $c$-vectors for acyclic cluster algebras and skew-symmetric cluster algebras of finite type obtained by  N\'{a}jera Ch\'{a}vez.
In a forthcoming paper~\cite{FuGengLiu}, the idea will be applied to algebras arising from cluster tubes to establish a link between dimension vectors of indecomposable $\tau$-rigid modules and $c$-vectors of cluster algebras of type $\mathrm{C}$.

The paper is organized as follows: After recall some definitions and basic properties related to $\tau$-tilting theory in  Section ~\ref{S:recollection}, we introduce the definition of $c$-vectors for finite-dimensional algebras in Section ~\ref{S:c-vectors and its sign-coherence}. We show that each $c$-vector can be realized as the (negative) dimension vector  of certain indecomposable module and establish the  sign-coherence property for $c$-vectors. Moreover, the relationship between positive $c$-vectors and negative $c$-vectors are also given. Section~\ref{S:c-vectos of quasitilted algebras} is devoted to study the $c$-vectors for quasitilted algebras. We show that the set $\opname{cv}^{+}(A)$ of positive $c$-vectors for a quasitilted algebra $A$ coincides  with the set $\opname{exdv}(A)$ of dimension vectors of exceptional $A$-modules.
This generalizes the equalities for acyclic cluster algebras established by  N\'{a}jera Ch\'{a}vez~\cite{Chavez13}. Let us mention here that,
these equalities also implies that an indecomposable $A$-module $M$ of  the quasitilted algebra $A$ can be completed to a $2$-term simple-minded collection of $\der^b(\mod A)$ if and only if $M$ is an exceptional module  (we refer to ~\cite{BY13} for the definition of $2$-term simple-minded collections).
 In Section~\ref{S:c-vectors of rep-directed} and ~\ref{s:cluster-tilted},  we also establish the equalities  between  the set  $\opname{cv}^{+}(A)$ of positive $c$-vectors  and the set  $\opname{exdv}(A)$ of dimension vectors of exceptional modules for representation-directed algebras and cluster-tilted algebras of finite type respectively.
 After recall the basic definitions of cluster algebras in Appendix \ref{S:appendix}, we give an explanation of cluster algebra of finite type via $c$-vectors and
 a short proof for the sign-coherence of $c$-vectors for skew-symmetric cluster algebras.

{\it Notations.} Throughout this paper,  $k$ denotes  an algebraically closed field, $A$ a finite-dimensional basic $k$-algebra.  All modules are right modules.
Let $\cc$ be a category over $k$, for an object $M\in \cc$, denote by $\add M$ the full subcategory of $\cc$ whose objects are direct summands of finite direct sum of $M$.

{\bf Acknowledgements.} The author thanks Dong Yang for pointing out the bijection between $2$-term silting objects and intermediate $t$-structures to prove the sign-coherence of $c$-vectors.  He is also grateful to Alfredo N\'{a}jera Ch\'{a}vez and Tomoki Nakanishi  for their interests and helpful comments.  This work was partially supported by NSFC (No.11471224) and the Fundamental Research Funds for the Central Universities (No.2013SCU04A44).

\section{Recollection}~\label{S:recollection}
In this section, we recall  some definitions and basic properties of (support) $\tau$-tilting modules and ($2$-term) silting objects. We mainly follow~\cite{AIR12,AI12,KY12}.
\subsection{Support $\tau$-tilting modules}
Let $A$ be a finite-dimensional algebra over  $k$ and  $\mod A$  the category of finitely generated right $A$-modules.
Let $S_1,\cdots, S_n$ be all the pairwise non-isomorphic simple $A$-modules and $P_1,\cdots, P_n$ the corresponding projective covers of $S_1,\cdots, S_n$ respectively.
Denote by $\tau$  the Auslander-Reiten translation  of $\mod A$.

An $A$-module $M$ is called {\it rigid} if $\Ext_A^1(M,M)=0$.
A module $M\in \mod A$ is called {\it $\tau$-rigid} provided  $\Hom_A(M,\tau M)=0$.  Let $P_1^M\xrightarrow{f}P_0^M\to M\to 0 $ be a minimal projective resolution of $M$, then $M$ is $\tau$-rigid if and only if $\Hom_A(f,M)$ is surjective.
Note that $\tau$-rigid implies rigid, but the converse is not true in general.

A {\it $\tau$-rigid pair} is $(M,P)$ with $M\in \mod A$ and $P$  a finitely generated projective $A$-module, such that $M$ is $\tau$-rigid and $\Hom_A(P,M)=0$.
A $\tau$-rigid pair is called {\it support $\tau$-tilting pair} if $|M|+|P|=n$, where $|X|$ denotes the number of non-isomorphic indecomposable direct summands of $X$.
In this case, $M$ is a {\it support $\tau$-tilting} $A$-module and $P$ is uniquely determined by $M$.

Recall that
a full subcategory $\ct$ of $\mod A$  is   a {\it torsion class} of $\mod A$ provided that $\ct$ is closed under quotients and extensions. An object $X\in \ct$ is {\it  $\Ext$-projective} if $\Ext^1_A(X,\ct)=0$.  A torsion pair $(\ct,\cf)$ is uniquely determined by its torsion class $\ct$ in the sense that \[\cf=^\perp \ct:=\{N\in \mod A|\Hom_A(M,N)=0~\text{for all}~M\in \ct\}.\]
A torsion pair $(\ct,\cf)$ is {\it functorially finite} provided that $\ct$ is functorially finite, equivalently, there is an object $X\in \mod A$ such that $\ct=\opname{Fac} X$, where $\opname{Fac}X$ is the subcategory of $\mod A$ formed by   quotients of finite direct sum of $X$.  Let  $P(\ct)$ be the direct sum of one copy of each of the indecomposable $\Ext$-projective objects in $\ct$ up to isomorphism. It is well-known that $\ct=\opname{Fac}P(\ct)$.
The following result due to ~\cite{AS81} will be used implicitly ({\it cf.} Theorem $5.10$ in ~\cite{AS81}).
\begin{proposition}
Let $A$ be a finite-dimensional algebra over $k$.   If $M$  is a $\tau$-rigid $A$-module, then  $\opname{Fac}M$ is a functorially finite torsion class  and $M\in \opname{Fac}M$ is $\Ext$-projective.
\end{proposition}

 Let $\opname{f-tors } A$ be the set of isomorphism classes of functorially finite torsion classes of $\mod A$ and $\opname{s\tau-tilt} A$ the set of isomorphism classes of basic support $\tau$-tilting $A$-modules.    The support $\tau$-tilting  $A$-modules  are closely related to the functorially finite torsion classes of $\mod A$.  In particular,  we have the following bijection ({\it cf.} Theorem $2.7$ of \cite{AIR12}).
\begin{theorem}~\label{t:bijection fftorsion-tautilting}
Let $A$ be a finite-dimensional algebra over $k$. There is a bijection between $\opname{s\tau-tilt} A$ and $\opname{f-tors} A$ given by
\[ M\in \opname{s\tau-tilt} A \mapsto \opname{Fac} M\in \opname{f-tors} A,
\]
and its inverse is given by $\ct\mapsto P(\ct)$, where $\ct\in \opname{f-tors}A$.
\end{theorem}

\subsection{Silting objects}
Let  $\der^b(\mod A)$  be the bounded derived category of finitely generated right $A$-modules with suspension functor $\Sigma$. Recall that $n$ is the number of pairwise non-isomorphic simple $A$-modules.
Let $\per A$ be the perfect derived  category of $A$, that is the smallest thick subcategory of $\der^b(\mod A)$ containing the object $A$.  An object $Q\in \per A$ is called {\it presilting} if $\Hom_{\per A}(Q,\Sigma^i Q)=0$ for all $i>0$. A presilting object $Q\in \per A$ is called a {\it silting object }provided moreover  $\opname{thick}(Q)=\per A$, where $\opname{thick}(Q)$ is the smallest thick subcategory of $\per A$ containing $Q$. Each basic silting object has  exactly $n$ indecomposable direct summands~\cite{AI12}.
A presilting object $Q$ is called {\it almost silting} if the number of non-isomorphic indecomposable direct summands of $Q$ is $n-1$. If there is an indecomposable object $X\in \per A$ such that $P\oplus X$ is a silting object, then $X$ is called a {\it complement }of $P$. In general, an almost presilting object may have infinite complements.

Let $Q=X\oplus P$ be a basic silting object with $X$ indecomposable. Consider the triangle
\[X\xrightarrow{f}Q_1\to Y\to \Sigma X,
\]
where $f$ is a minimal left $\add P$-approximation of $X$. It has been shown in~\cite{AI12} that $Y\oplus P$ is a basic silting object and called {\it the left mutation} of $Q$ with respect to $X$. Dually, if we consider the triangle induced by a minimal right $\add P$-approximation of $X$, we obtain the right mutation of $Q$ with respect to $X$.

A silting object $Q\in \per A$ is {\it 2-term silting} if there is a triangle \[P_1^Q\to P_0^Q\to Q\to \Sigma P_1^Q, ~\text{where}~ P_0^Q,P_1^Q\in \add A.
\]
 Denote by $\opname{2-silt} A$ the set of isomorphism classes of 2-term silting objects of $\per A$.
The following   has been established in ~\cite{AIR12}.
\begin{theorem}~\label{t:bijection silting-tautilting}
Let $A$ be a finite-dimensional algebra over $k$.
\begin{itemize}
\item[(1)] Let $P$ be an almost $2$-term silting object in $\per A$, there  exists exactly two indecomposable objects $X,Y$ such that $P\oplus X$ and $P\oplus Y$ are $2$-term silting objects in $\per A$;  Moreover, $P\oplus X$ and $P\oplus Y$ are related by a left or right mutation;
 \item[(2)]There is a bijection between $\opname{s\tau-tilt} A$ and $\opname{2-silt} A$ given by
\[M\in \opname{s\tau-tilt} A\mapsto (P_1^M\oplus P\xrightarrow{(f,0)}P_0^M)\in \opname{2-silt} A,
\]
where $P_1^M\xrightarrow{f}P_0^M\to M$ is a minimal projective resolution of $M$ and $(M,P)$ is the  support $\tau$-tilting pair.
\end{itemize}
\end{theorem}

\subsection{$t$-structures on triangulated categories}
Let $\der$ be a triangulated category over $k$ with suspension functor $\Sigma$. A pair of full subcategory $(\cu^{\leq 0}, \cv^{{\geq 0}})$ of $\der$ is called a {\it  $t$-structure} on $\der$ provided that
\begin{itemize}
\item[(1)] $\Sigma \cu^{{\leq 0}}\subseteq \cu^{\leq 0}$;
\item[(2)] $\Hom_{\der}(\cu^{\leq 0}, \Sigma^{-1}\cv^{\geq 0})=0$;
\item[(3)]for each $X\in \der$, there is a triangle $U_X\to X\to V_X\to \Sigma U_X$ with $U_X\in \cu^{\leq 0}$ and $V_X\in \Sigma^{-1}\cv^{\geq 0}$.
 \end{itemize}
A {\it bounded $t$-structure} on $\der$ is a $t$-structure $(\cu^{\leq 0}, \cv^{\geq 0})$ such that \[\der=\bigcup_{n\in \Z}\Sigma^n\cu^{\leq 0}=\bigcup_{n\in \Z}\Sigma^n\cv^{\geq 0}.\]

For a given $t$-structure $(\cu^{\leq 0}, \cv^{\geq 0})$ on $\der$,   the subcategory $\ca=\cu^{\leq 0}\cap \cv^{\geq 0}$ of $\der$ is called {\it the heart } of $(\cu^{\leq 0},\cv^{\geq 0})$, which is an abelian category with the exact structure induced by the triangles of $\der$. Moreover, for any $X,Y\in \ca$, we have $\Hom_\ca(X,Y)=\Hom_\der(X,Y)$ and $\Ext^1_{\ca}(X,Y)=\Hom_\der(X,\Sigma Y)$.
The triangles in $(3)$ are canonical and  yield endofunctors $\tau_{\leq 0}$ and $\tau_{\geq 1}$ of $\der$ such that $\tau_{\leq 0}X=U_X$ and $\tau_{\geq 1}X= V_X$.
 The functors $\tau_{\leq 0}$ and $\tau_{\geq 1}$ give rise to a family  of cohomological functors $H^i=\tau_{\leq i}\circ \tau_{\geq i}:\der \to \ca$, where $\tau_{\leq i}=\Sigma^{-i}\circ \tau_{\leq 0}\circ \Sigma^i$ and $\tau_{\geq i}=\Sigma^{-i+1}\circ \tau_{\geq 1}\circ \Sigma^{i-1}$. Moreover, for each $X\in \der$, we have a family of triangles
\[\tau_{\leq i}X\to X\to \tau_{{\geq i+1}}X\to \Sigma \tau_{\leq i}X, ~\text{where}~i\in \Z.
\]

\subsection{Negative  dg algebra associated to a silting object}~\label{s:negative dg algebra}
Recall that $A$ is a finite-dimensional $k$-algebra. Let $T=\oplus_{i=1}^nT_i\in \per A$ be a basic silting object with indecomposable direct summands  $T_1, \cdots, T_n$ and $\widetilde{\Gamma}=\RHom_{A}(T,T)$ the dg endomorphism algebra of $T$.
By the definition of silting object, we know that the homology groups $H^{i}(\widetilde{\Gamma})$ vanish for all $i>0$. Denote by $\Gamma=\tau_{\leq 0}\widetilde{\Gamma}$  the truncation algebra  of $\widetilde{\Gamma}$.   Let $\lambda:\Gamma\to \widetilde{\Gamma}$ be the canonical injective homomorphism of dg algebras. It is clear that $\lambda$ induces an equivalence of derived categories $\der(\Gamma)\cong \der(\widetilde{\Gamma})$.
On the other hand, we also  have the surjective homomorphism  $\pi:\Gamma\to H^{0}(\Gamma)\cong \End_A(T)$ of dg algebras,  where $\End_{A}(T)$ is the endomorphism algebra of $T$ in $\per A$.

Let $e_i=1_{T_i}\in \Hom_{ A}(T_i,T_i), 1\leq i\leq n,$  be the primitive orthogonal idempotents in $\End_{A}(T)$, which will induce a decomposition  of the identity of $\Gamma$ into a sum of primitive idempotents. By abuse of notations, we  still denote the corresponding primitive idempotents by $e_1, \cdots, e_n$. Thus we have the decomposition of $\Gamma=\oplus_{i=1}^ne_i\Gamma$ into indecomposable right $\Gamma$-modules.  Moreover, the images $[e_{1}\Gamma], \cdots,[e_{n}\Gamma]$ form a $\Z$-basis of the Grothendieck group $\go(\per \Gamma)$ of the perfect derived category $\per \Gamma$ of $\Gamma$.
Let $S_{1}^{T}, \cdots, S_{n}^{T}$ be pairwise non-isomorphic simple right $\End_{A}(T)$-modules. Via the homomorphism $\pi$, each simple $\End_{A}(T)$-module $S_{i}^{T}$ lifts to a simple dg $\Gamma$-module $S_{i}^{\Gamma}$. Let $\der_{fd}(\Gamma)$ be the finite-dimensional derived category of $\Gamma$, that is the full triangulated subcategory of $\der(\Gamma)$ formed by the dg $\Gamma$-modules whose homology has finite total dimension over $k$.
Similarly, the images $[S_{1}^{\Gamma}], \cdots, [S_{n}^{\Gamma}]$ form a $\Z$-basis of the Grothendieck group $\go(\der_{fd}(\Gamma))$ of the finite-dimensional derived category $\der_{fd}(\Gamma)$ of $\Gamma$.
  Let $\langle-,-\rangle_{\Gamma}:\go(\per \Gamma)\times \go(\der_{fd}(\Gamma))\to k$ be the non-degenerate  Euler bilinear form given by
  \[\langle[P], [X]\rangle_{\Gamma}=\sum_{i\in \Z}\dim_{k}\Hom_{\Gamma}(P,\Sigma^{i }X),
  \]
  where $P\in \per \Gamma$ and $X\in \der_{fd}(\Gamma)$.
 For any $X\in \der(\Gamma),t\in \Z$,  we clearly have
 \[\Hom_{\Gamma}(\Gamma, \Sigma^{t}X)=\begin{cases}k& t=0;\\0&~\text{otherwise}. \end{cases}\]
 if and only if there exists a unique $i$ such that $X\cong S_{i}^{\Gamma}$ in $\der(\Gamma)$. In other words, $[e_{1}\Gamma], \cdots, [e_{n}\Gamma]$ and $[S_{1}^{\Gamma}], \cdots, [S_{n}^{\Gamma}]$ are dual bases with respect to the Euler bilinear form $\langle-,-\rangle_{\Gamma}$.

Recall that $T$ is a silting object in $\per A$, we have an equivalence $\der(\widetilde{\Gamma})\cong\der(\Mod A)$ and hence an equivalence between $\der(\Gamma)$ and $\der(\Mod A)$. Indeed, view $T$ as $\Gamma^{\opname{op}}\otimes_{k}A$-module, the equivalence is given by $F:=\lten_{\Gamma} T:\der(\Gamma)\to \der(\Mod A)$, which restricts to equivalences $\per \Gamma\cong\per A$ and $\der_{fd}(\Gamma)\cong \der^b(\mod A)$ respectively.

Since $\Gamma$ is a finite-dimensional negative dg algebra, there is a standard $t$-structure $(\der^{\leq 0}, \der^{\geq 0})$ on $\der_{fd}(\Gamma)$ induced by the homology whose heart is equivalent to $\mod \End_{A}(T)$. More precisely, \[\der^{\leq 0}=\{X\in \der_{fd}(\Gamma)|\Hom_{\Gamma}(\Gamma,\Sigma^{i}X)=0~\text{for all} ~i>0\},\]
\[~\der^{\geq 0}=\{X\in \der_{fd}(\Gamma)|\Hom_{\Gamma}(\Gamma,\Sigma^{i}X)=0~\text{for all}~i<0\}.
\]
The standard $t$-structure $(\der^{\leq 0}, \der^{\geq 0})$ induces a $t$-structure on $\der^b(\mod A)$ via the functor $F$. Denote by $(\der^{\leq 0}_T, \der^{\geq 0}_T)$ the resulting $t$-structure, that is
 \[\der_T^{\leq 0}=\{X\in \der^b(\mod A)|\Hom_A(T, \Sigma^i X)=0~\text{for all } i>0\},
 \]
 \[\der_T^{\geq 0}=\{X\in \der^b(\mod A)|\Hom_A(T, \Sigma^i X)=0~\text{for all} ~i<0\}.
 \]
  Let $\ca=\der_T^{\leq 0}\cap\der_T^{\geq 0}$ be the heart of the $t$-structure  $(\der_T^{\leq 0}, \der_T^{\geq 0})$. It is clear that $F(S_1^\Gamma), \cdots, F(S_n^\Gamma)$ are all the simple objects of $\ca$.
     If $T$ is a $2$-term silting object, by Theorem~\ref{t:bijection silting-tautilting} and \ref{t:bijection fftorsion-tautilting}, there is a support $\tau$-tilting module and a functorially finite torsion class corresponding to $T$ respectively.
 We  have the following characterization of $\ca$ by torsion pair, which is a consequence of the bijections investigated in ~\cite{KY12} ({\it cf.} also ~\cite{BY13}), for completeness and later use, we include a  proof.
 \begin{proposition}~\label{p:torsion pair}
 Keep the notations above.
Assume that $T$ is a 2-term silting object and $M\in \mod A$ is the associated support $\tau$-tilting $A$-module. Let $\ct_M=\opname{Fac} M$ be the functorially  finite torsion class associated to $M$ and $\cf_M=^\perp\ct_M$ the torsion free class. Then $(\Sigma \cf_M, \ct_M)$ is a torsion pair of $\ca$. As a consequence, each simple object of $\ca$ lies either in $\Sigma \cf_M$ or in $\ct_M$.
 \end{proposition}
 \begin{proof} Let $P_1^M\xrightarrow{f}P_0^M\to M\to 0$ be a minimal projective resolution of $M$ and $(M, Q)$  the associated basic support $\tau$-tilting pair. We have \[T= \cdots\to0\to P_1^M\oplus Q\xrightarrow{(f,0)}P_0^M\to 0\cdots,\] where $P_0^M$ is in the zeroth component.
  Note that $X\in \ca$ if and only if $\Hom_A(T,\Sigma^i X)=0$ for $i\neq 0$. A direct calculation shows that  $\Sigma \cf_M\subset \ca$ and $\ct_M\subset \ca $.
 On the other hand, we clearly have $\Hom_A(\Sigma \cf_M, \ct_M)=0$. Note that the exact sequences of $\ca$ are induced from the triangles of $\der^b(\mod A)$.
 Thus to prove $(\Sigma \cf_M, \ct_M)$ is a torsion pair of $\ca$, it remains to show that for each $X\in \ca$, there is a triangle $\Sigma F_0\to X\to T_0\to \Sigma^2F_0$ in $\der^b(\mod A)$ with $T_0\in \ct_M$ and $F_0\in \cf_M$.

  Let $(\mathcal{C}^{\leq 0}, \mathcal{C}^{\geq 0})$ be the standard $t$-structure on $\der^b(\mod A)$ and $H^i$ the associated cohomological functors. We claim that if $X\in \ca\subset \der^b(\mod A)$, then   $H^i(X)=0$ for $i\neq 0,-1$.
 For any $X\in\ca\subset \der^b(\mod A)$, consider the following triangle
 \[\tau_{\leq -2}X\to X\to \tau_{\geq -1}X\to \Sigma \tau_{\leq -2}X\]
 induced by the standard $t$-structure $(\mathcal{C}^{\leq 0}, \mathcal{C}^{\geq 0})$.
Applying the functor $\Hom_A(T,?)$ yields a long exact sequence
\[\cdots\to \Hom_A(T, \Sigma^{i}\tau_{\leq -2}X)\to \Hom_A(T, \Sigma^{i}X)\to \Hom_A(T, \Sigma^{i}\tau_{\geq-1}X)\to   \cdots.\]
We have $\Hom_A(T, \Sigma^i \tau_{\leq -2}X)=0$ for all $i$ since $\Hom_A(T, \Sigma^iX)=0$ for all $i\neq 0$. Recall that we have $\opname{thick}(T)=\per A$, which implies that  $\tau_{\leq -2}X=0$ in $\der^b(\mod A)$. As a consequence,  $X\cong \tau_{\geq-1}X\in \Sigma\mathcal{C}^{\geq 0}$. Now consider the triangle
\[\tau_{\leq 0}X\to X\to \tau_{\geq 1}X\to \Sigma \tau_{\leq 0}X,
\]
applying  the functor $\Hom_A(T,?)$ to the triangle yields a long exact sequence
\[\cdots\to \Hom_A(T,\Sigma^i\tau_{\leq 0}X)\to \Hom_A(T,\Sigma^iX)\to \Hom_A(T,\Sigma^i \tau_{\geq 1}X)\to \cdots.
\]
Again one can show that $\Hom_A(T,\Sigma^i\tau_{\geq 1}X)=0$ for all $i$, and hence $\tau_{\geq 1}X=0$ in $\der^b(\mod A)$. In particular, we have proved that $X\cong \tau_{\leq 0}\circ \tau_{\geq -1}X\in \mathcal{C}^{\leq 0}\cap \Sigma\mathcal{C}^{\geq 0}$, which implies that $H^i(X)=0$ for $i\neq 0$ or $-1$.

By the  standard $t$-structure $(\cc^{\leq 0},\cc^{\geq 0})$, for each $X\in \ca$, we have the following triangle in $\der^b(\mod A)$
\[\Sigma H^{-1}(X)\to X\to H^0(X)\to \Sigma^2 H^1(X).
\]
 It remains to show that $H^0(X)\in \ct_M$ and $H^{-1}(X)\in \cf_M$ for $X\in \ca$. It is easy to see that $\Hom_A(T, \Sigma^i H^0(X))=0$ for all $i\neq 0$ and $\Hom_A(T, \Sigma^iH^{-1}(X))=0$ for all $i\neq 1$. Consider the short exact sequence \[0\to T_{H^0(X)}\to H^0(X)\to F_{H^0(X)}\to 0\] induced by the torsion pair $(\ct_M,\cf_M)$ in $\mod A$ with $T_{H^0(X)}\in \ct_M$ and $F_{H^0(X)}\in \cf_M$. Applying $\Hom_A(T,?)$ to the exact sequence, one can show that $\Hom_A(T, \Sigma F_{H^0(X)})=0$. Recall that we also have $\Hom_{A}(T,\Sigma^{i}F_{H^{0}(X)})=0$ for all $i\neq 1$. Consequently,  $F_{H^0(X)}=0$ in $\der^b(\mod A)$. In particular, we have $T_{H^0(X)}\cong H^0(X)\in \ct_M$. Similarly, one can show that $H^{-1}(X)\in \cf_M$. This finishes the proof.
 \end{proof}

\section{ c-vectors and its sign-coherence}~\label{S:c-vectors and its sign-coherence}
\subsection{Definition of $c$-vectors}
Recall that $A$ is a finite-dimensional algebra over $k$ and $n$ is the number of non-isomorphic simple $A$-modules.
Let $\go^{\text{sp}}(\add A)$ be the split Grothendieck group of finitely generated projective $A$-modules.
For a given $\tau$-rigid $A$-module $M$, let $P_1^M\xrightarrow{f}P_0^M\to M\to 0$ be a minimal projective resolution of $M$, the {\it index}  of $M$ is defined to be  $\opname{ind}(M)=[P_0^M]-[P_1^M]\in \go^{\text{sp}}(\add A)$.
The {\it $g$-vector}  of $M$ is $g(M)=(g_1,\cdots, g_n)'\in \Z^n$ with $g_i=[\opname{ind}(M):P_i], 1\leq i\leq n$.
It has been proved in ~\cite{AIR12} that different $\tau$-rigid $A$-modules have different indices and hence different $g$-vectors.

For a given  basic support $\tau$-tilting pair $(M,P)$ with decomposition of indecomposable modules $M=\bigoplus_{i=1}^tM_i, P=\bigoplus_{i=t+1}^nP_i^M$, we have the following {\it  $\opname{G}$-matrix } of $(M,P)$
\[\opname{G}_{(M,P)}=(g(M_1), g(M_2),\cdots, g(M_t), -g(P_{t+1}^M),\cdots, -g(P_n^M))\in M_n(\Z).
\]
We know from ~\cite{AIR12} that for any basic support $\tau$-tilting pair $(M,P)$ the $G$-matrix $\opname{G}_{(M,P)}$  is invertible over $\Z$.
Inspired by the tropical dualities between $g$-vectors and $c$-vectors in cluster algebras, we introduce the $\opname{C}$-matrix of a basic support $\tau$-rigid pair $(M,P)$ to be the inverse of the transpose of the $\opname{G}$-matrix $\opname{G}_{(M,P)}$, {\it i.e.}
\[\opname{C}_{(M,P)}:=(\opname{G}_{(M,P)}^T)^{-1}\in M_n(\Z).
\]
Each column vector of $\opname{C}_{(M,P)}$ is called a {\it $c$-vector} of $A$ and denote by $\opname{cv}(A)$ the set of all  the $c$-vectors of $A$.

\subsection{Sign-coherence of $c$-vectors}
A vector $c$ in $\Z^n$ is called {\it sign-coherence} if $c$ has either all  entries nonnegative or  all entries nonpositive. A non-zero vector in $\Z^n$ is {\it positive} (resp.
{\it negative}) if all components are nonnegative (resp. nonpositive). The sign-coherence phenomenon holds for this general setting.
\begin{theorem}~\label{t:sign-coherence}
Let $A$ be a finite-dimensional algebra over  $k$. Then each $c$-vector of $A$ is sign-coherence.
\end{theorem}
\begin{proof}
Let $S_1,\cdots, S_n$ be all the pairwise non-isomorphic simple $A$-modules and $P_1,\cdots, P_n$ the corresponding projective covers of $S_1,\cdots, S_n$ respectively. Let  $\go(\per A)$  and $\go(\der^{b}(\mod A))$   be the Grothedieck groups of   $\per A$ and $\der^{b}(\mod A)$   respectively. Denote by $\langle-,-\rangle_A:\go(\per A)\times \go(\der^b(\mod A))\to k$ the Euler bilinear form given by  $\langle [P], [X]\rangle_A=\sum_{i\in \Z}\dim_k\Hom_A(P,\Sigma^i X)$ for any $P\in \per A $ and $X\in \der^b(\mod A)$. It is clear that $[P_{1}],\cdots, [P_{n}]$ and $[S_{1}], \cdots, [S_{n}]$ are dual bases with respect to the Euler bilinear form $\langle-,-\rangle_{A}$.

Let $(M,Q)$ be a basic support $\tau$-tilting pair  of $A$ and $T$ the corresponding $2$-term silting object in $\per A$.
Let $\Gamma=\oplus_{i=1}^{n}e_{i}\Gamma$ be the negative truncated  dg algebra associated to $T$ ({\it cf.} Section~\ref{s:negative dg algebra}). Recall that we also have the Euler bilinear form $\langle-,-\rangle_{\Gamma}:\go(\per \Gamma)\times \go(\der_{fd}(\Gamma))\to k$ and there is an equivalence of triangulated categories $F=\lten_{\Gamma}T_A:\der_{fd}(\Gamma)\to \der^{b}(\mod A)$.
 It is clear that the  functor $F$ induces an isomorphism  of bilinear forms such that the following diagram is commutative
\[\xymatrix{\langle-,-\rangle_{\Gamma}:\go(\per \Gamma)\times \go(\der_{fd}(\Gamma))\ar[d]^{F}\ar[drr]\\
\langle-,-\rangle_A:\go(\per A)\times \go(\der^b(\mod A))\ar[r] & &k.}
\]
Note that the column vectors the $G$-matrix associated to $(M,Q)$ is the dimension vector of $F(e_i\Gamma)$ in $\go(\per A))$ with respect to the basis $[P_{1}], \cdots,[P_{n}]$.
 By the duality between $G$-matrix and $C$-matrix, we deduce that  the  $c$-vectors associated to $(M,Q)$ are the dimension vectors of $F(S_{i}^{\Gamma})$ for all the simple dg $\Gamma$-module $S_{i}^{\Gamma}$.  Now the result follows from Proposition ~\ref{p:torsion pair}.
\end{proof}

As a byproduct of the proof, we have the following criterion  of $c$-vectors.
\begin{proposition}~\label{p:criterion}
Let $A$ be a finite-dimensional algebra over $k$. A vector $c\in \Z^n$ is a $c$-vector of $A$ if and only if there is a $2$-term silting object $T\in \per A$ and an indecomposable $A$-module $M$ satisfying one of the following conditions:
\begin{itemize}
\item[(1)]$\Hom_{ A}(T, \Sigma^i M)= \begin{cases}k & i=0;\\ 0& \text{otherwise}.\end{cases}$ and  $c=\dimv M$;
\item[(2)]$\Hom_{ A}(T, \Sigma^i M)=\begin{cases}k & i=1;\\ 0 & \text{otherwise}.\end{cases}$ and  $c=-\dimv M$.
\end{itemize}
\end{proposition}

\subsection{Positive $c$-vectors and negative $c$-vectors}
Let $\opname{cv}^+(A)$ be the set of positive $c$-vectors of $A$ and $\opname{cv}^-(A)$ the set of negative $c$-vectors.
The following result is a consequence of the bijection between  $2$-term silting objects and $2$-term simple-minded collections investigated in ~\cite{BY13}. In order to avoid  more notations, we give a proof using the mutation of silting objects.
\begin{proposition}~\label{t:equality}
Let $A$ be a finite-dimensional algebra over $k$. We have $\opname{cv}^{-}(A)=-\opname{cv}^+(A)$. In particular, $\opname{cv}(A)=-\opname{cv}^+(A)\cup\opname{cv}^+(A)$.
\end{proposition}
\begin{proof}
We show the inclusion $-\opname{cv}^+(A)\subseteq \opname{cv}^-(A)$. The inverse inclusion is similar.
Let $c$ be an arbitrary positive $c$-vector of $A$. Then there is a $2$-term silting object, say $T\in \per A$ and an indecomposable $A$-module $M$ such that $\Hom_A(T, M)=k$ and $\Hom_A(T, \Sigma^i M)=0$ for all $i\neq 0$. We may rewrite $T$ as $T=T_M\oplus Q$ with $T_M$ indecomposable such that $\Hom_A(T_M, M)=k, \Hom_A(T_M, \Sigma^iM)=0$ for $i\neq 0$ and $\Hom_A(Q, \Sigma^iM)=0$ for all $i\in \Z$. It is known that there is an indecomposable $2$-term presilting, say  $T_N$, such that $T'=T_N\oplus Q$ is a basic $2$-term silting object in $\per A$. By  Theorem~\ref{t:bijection silting-tautilting} $(1)$, we know that $T$ and $T'$ are related by a left or right mutation. We claim that $T'$ is the left mutation of $T$. Otherwise, $T$ is the left mutation of $T'$ and we have the triangle $T_N\to Q_1\to T_M\to \Sigma T_N$ with $Q_1\in \add Q$. Applying the functor $\Hom_A(?,M)$, we have a long exact sequence
\[\cdots\Hom_A(T_M, \Sigma^i M)\to \Hom_A(Q_1, \Sigma^i M)\to \Hom_A(T_N, \Sigma^i M)\to \Hom_A(T_M, \Sigma^{i+1}M)\cdots,
\]
which implies that $\Hom_A(T', \Sigma^{-1}M)=k$ and $\Hom_{A}(T', \Sigma^{i}M)=0$ for all $i\neq -1$. Let $\Gamma_{T'}$ be the negative dg algebra associated to $T'$. The conditions $\Hom_A(T', \Sigma^{-1}M)=k$ and $\Hom_{A}(T', \Sigma^{i}M)=0$ for all $i\neq -1$ imply that $\RHom_A(T', \Sigma^{-1}M)$ is a simple $\Gamma_{T'}$-module, which contradicts to Proposition~\ref{p:torsion pair}.

Therefore $T'$ has to be the  left mutation of $T$ and there is a triangle
 \[ T_M\to Q_2\to T_N\to \Sigma T_M, \text{where}~ Q_2\in \add Q.
\]
Applying the functor $\Hom_A(?, M)$, we obtain a long exact sequence
\[\cdots\Hom_A(T_N, \Sigma^i M)\to \Hom_A(Q_2, \Sigma^i M)\to \Hom_A(T_M, \Sigma^i M)\to \Hom_A(T_N, \Sigma^{i+1}M)\cdots.
\]
We have $\Hom_A(T_N, \Sigma M)=k$ and $\Hom_A(T_N, \Sigma^iM)=0$ for all $i\neq 1$. As a consequence, $\Hom_A(T', \Sigma M)=k$ and $\Hom_A(T', \Sigma^i M)=0$ for all $i\neq 1$. In particular, $-\dimv M$ is a negative $c$-vector by Proposition~\ref{p:criterion}   and we have $-\opname{cv}^+(A)\subseteq \opname{cv}^-(A)$.
\end{proof}

\subsection{The left-right symmetry of $c$-vectors}
Let $A^{\opname{op}}$ be the opposite $k$-algebra of $A$.  We have the dualities
\[D=\Hom_{k}(?,k): \mod A\to \mod A^{\opname{op}} ~\text{and }~(-)^{*}=\Hom_{A}(?,A):\add A\to \add  A^{\opname{op}}.
\]
For any $X\in \mod A$,  let \[P_{1}\xrightarrow{d} P_{0}\to X\to 0\]be a minimal projective resolution of $X$, its transpose $\opname{Tr}X\in \mod A^{\opname{op}}$ is defined by the following exact sequence
\[P_{0}^{*}\xrightarrow{d^{*}}P_{1}^{*}\to \opname{Tr}X\to 0.
\]
For any $M\in \mod A$, we can decompose $M$ as $M=M_{pr}\oplus M_{np}$, where $M_{pr}$ is a maximal projective direct summand of $M$. The following left-right symmetry of $\tau$-rigid modules has been established in ~\cite{AIR12} ({\it cf.} Theorem 2.14 of \cite{AIR12}).
\begin{theorem}~\label{t:left-rightsymmetry}
Let $A$ be a finite-dimensional $k$-algebra. There is a bijection $(-)^{\circ}$ between $\opname{s\tau-tilt}A$ and $\opname{s\tau-tilt}A^{\opname{op}}$ given by
$(M,Q)^{\circ}=(\opname{Tr}M_{np}\oplus Q^{*}, M_{pr}^{*} )$, where $(M,Q)$ is a support $\tau$-tilting pair of $A$. Moreover, $(-)^{\circ\circ}=\opname{id}$.
\end{theorem}

Let $M$ be an indecomposable non-projective $\tau$-rigid $A$-modules. By Theorem~\ref{t:left-rightsymmetry}, we infer that $\opname{Tr}M$ is also $\tau$-rigid as $A^{\opname{op}}$-module. Moreover, we clearly have $g(M)=-g( \opname{Tr}M)$. On the other hand, for any indecomposable projective $A$-module $P$, we also have $g(P)=-g(P^{*})$.
Now the following result is an immediate consequence of the definition of $c$-vectors and Theorem~\ref{t:left-rightsymmetry}, Theorem~\ref{t:sign-coherence} and Proposition \ref{t:equality}.
\begin{proposition}~\label{p:equality opposite}
Let $A$ be a finite-dimensional $k$-algebra and $A^{\opname{op}}$ its opposite algebra. Then we have $\opname{cv}(A)=-\opname{cv}(A^{\opname{op}})$ and $\opname{cv}^{+}(A)=\opname{cv}^{+}(A^{\opname{op}})$.
\end{proposition}

\section{c-vectors vs dimension vectors}
For an algebra $A$, let $\opname{dv}(A)$ be the set of dimension vectors of indecomposable $A$-modules. By Proposition~\ref{p:criterion}, we know that each positive $c$-vector can be realized as the dimension vector of an indecomposable $A$-module, that is, $\opname{cv}^{+}(A)\subseteq \opname{dv}(A)$.  However, the inverse inclusion is not true in general. The aim of this section is to study the positive $c$-vectors for  quasitilted algebras, representation-directed algebras and cluster-tilted algebras of finite type. In particular, we recover the equalities of $c$-vectors for acyclic cluster algebras and skew-symmetric cluster algebras of finite type respectively obtained by N\'{a}jera Ch\'{a}vez.

\subsection{$c$-vectors of quasitilted algebras}~\label{S:c-vectos of quasitilted algebras}
\subsubsection{Hereditary abelian categories}
We follow~\cite{Lenzing07}.
Throughout this section, let $\mathcal{H}$ be a hereditary abelian $k$-category with finite-dimensional morphism and extension spaces. As a consequence of finite-dimensional morphism space, $\mathcal{H}$ is a Krull-Schmidt category, {\it i.e.} each object of $\mathcal{H}$ is a finite direct sum of indecomposable objects with local endomorphism ring. We refer to ~\cite{Lenzing07} for examples and basic properties of hereditary categories.

Let $\der^{b}(\mathcal{H})$ be the bounded derived category of $\mathcal{H}$ with the suspension functor $\Sigma$.
An object $X$ in $\der^{b}(\mathcal{H})$ is called {\it rigid } provided  $\Hom_{\der^{b}(\mathcal{H})}(X,\Sigma X)=0$.  A rigid object $X$ is {\it exceptional} if $\dim_{k}\Hom_{\der^{b}(\ch)}(X,X)=1$.
The following fundamental result is due to Happel-Ringel~\cite{HR82} ({\it  cf.} also \cite{ASS06, Lenzing07}).
\begin{lemma}~\label{l:Happel-Ringel}
Let $E$ and $F$ be indecomposable objects in $\mathcal{H}$ such that $\Hom_{\der^{b}(\ch)}(F,\Sigma E)=0$. Then all non-zero homomorphism $f:E\to F$ is a monomorphism or epimorphism.  In particular, each indecomposable $E$ without self-extensions is exceptional.
\end{lemma}

Let $\cc$ be a full subcategory of $\der^{b}(\ch)$ and $M$ an indecomposable object in $\cc$. A path in $\cc$ from $M$ to itself is called a {\it cycle} in $\cc$, that is a sequence of non-zero non-isomorphism between indecomposable objects in $\cc$ of the form
\[M=M_0\xrightarrow{f_1}M_1\xrightarrow{f_2}M_2\cdots\xrightarrow{f_r}M_r=M.
\]
The following result is a consequence of Lemma~\ref{l:Happel-Ringel}, which is crucial for our  investigation of $c$-vectors for quasitilted algebras.
\begin{lemma}~\label{l:nocycle}
Let $T$ be an object in $\der^{b}(\ch)$ such that $\Hom_{\der^{b}(\ch)}(T,\Sigma T)=0$. Then the subcategory $\add T$ has no cycle.
\end{lemma}
\begin{proof}
Suppose that there is a cycle in $\add T$, say $M=M_{0}\xrightarrow{f_{1}}M_{1}\xrightarrow{f_{2}}\cdots\xrightarrow{f_{r}}M_{r}=M$, where $M_1,\cdots, M_{r} $ are indecomposable objects in $\add T$. We may assume that $M_{0}\in \mathcal{H}$. Note that $M_{0}$ is exceptional and  $f_{1}$ is non-zero non-isomorphism, which imply that $M_{1}\not\cong M_{0}$.
We claim that some of $M_{1},\cdots, M_{r-1}$ are not in $\mathcal{H}$. Otherwise,
by Lemma~\ref{l:Happel-Ringel}, each  $f_{i}$ is either monomorphism or epimorphism.  If there is an epimorphism $f_{i}$ followed by a monomorphism $f_{i+1}$, then $f_{i+1}\circ f_{i}:M_{i-1}\to M_{i+1}$ is non-zero and is neither a monomorphism nor an epimorphism, which contradicts to Lemma~\ref{l:Happel-Ringel}. On the other hand, if there is a monomorphism $f_{i}$ followed by an epimorphism $f_{i+1}$, we may consider the cycle
$M_{i}\xrightarrow{f_{i+1}}\cdots M_{r}\xrightarrow{f_{1}}M_{1}\xrightarrow{f_{2}}\cdots \xrightarrow{f_{i-2}}M_{i-1}\xrightarrow{f_{i}}M_{i}$. This cycle turns out to admit an epimorphism followed by a monomorphism, a contradiction.
Hence all of the $f_{i}$ are either epimorphisms or monomorphisms. Note that $\dim_{k}\Hom_{\der}(M_{0},M_{0})=1$, we deduce that $f_{r}\circ\cdots\circ f_{1}$ is an isomorphism. Consequently, $f_{1}$ is an isomorphism, a contradiction.

Note that $\mathcal{H}$ is hereditary, $\Hom_{\der^{b}(\ch)}(\mathcal{H}, \Sigma^{i}\mathcal{H})=0$ for all $i\neq 0,1$. Since we have assume that $M_{0}$ belongs to $\mathcal{H}$ and all $f_{i}$ are non-zero, we deduce that some of $M_{1}, \cdots, M_{r-1}$ belong to $\Sigma^{t} \mathcal{H}$ for certain $t>0$. In particular, $M_{r-1}$ belongs to $\Sigma^{m}\mathcal{H}$ for some $m>0$. Consequently, $\Hom_{\der}(M_{r-1}, M_{0})=0$, which contradicts to $f_{r}\neq 0$.
\end{proof}

\begin{corollary}~\label{c:nocycle}
Let $A$ be a finite-dimensional $k$-algebra such that $\mod A$ is equivalent to the heart of a bounded $t$-structure on $\der^{b}(\ch)$. Then the Gabriel quiver $Q_{A}$ of $A$ has no cycle.
\end{corollary}
\begin{proof}
Let $\ca$ be the heart of a bounded $t$-structure on $\der^{b}(\ch)$ which is equivalent to $\mod A$. We consider $A$ as an object in $\der^{b}(\ch)$ via the equivalence.  We clearly have
$\Hom_{\der^{b}(\ch)}(A,\Sigma A)=\Ext^{1}_{A}(A,A)=0$. Note that any cycle in the Gabriel quiver $Q_{A}$ induces an oriented cycle in $\add A$. Now the result follows from Lemma~\ref{l:nocycle}.
\end{proof}

For any finite-dimensional $k$-algebra $A$, set
\[\opname{\tau dv}(A):=\{\dimv M|M\in \mod A~\text{is indecomposable $\tau$-rigid}\}.
\]
The following result gives a lower bound of $c$-vector for algebras related to the hearts of bounded $t$-structures on $\der^{b}(\ch)$.

\begin{proposition}~\label{p:lowerbound}
Let $A$ be a finite-dimensional $k$-algebra such that $\mod A$ is equivalent to the heart of a bounded $t$-structure on $\der^{b}(\ch)$. Then we have $\opname{\tau dv}(A)\subseteq \opname{cv}^+(A)$.  \end{proposition}
\begin{proof}
We need to prove that for any indecomposable $\tau$-rigid $A$-module $M$,  $\dimv M\in \opname{cv}^+(A)$.  Since $M$ is $\tau$-rigid, the subcategory $\opname{Fac}M$ is a functorially finite torsion class. Let $P=P(\opname{Fac}M)$ be the direct sum of one copy of each of indecomposable $\Ext$-projective objects in $\opname{Fac}M$. We may write $P=M\oplus M_1\oplus \cdots \oplus M_r$. Note that $P$ is a support $\tau$-tilting $A$-module, hence $\Hom_{\der^b(\ch)}(P,\Sigma P)=\Ext^1_A(P,P)=0$. On the other hand, by the definition of $\opname{Fac}M$, we deduce that $\Hom_{\der^{b}(\ch)}(M,M_i)=\Hom_A(M,M_i)\neq 0$ for any $i$. Now Lemma ~\ref{l:nocycle} implies that $\Hom_\der(M_i,M)=0$ for all $i$. Let $T$ be the $2$-term silting object in $\der^b(\mod A)$ corresponding to $P$.  A direct calculation shows  that $\Hom_{\der^{b}(\ch)}(T,M)=k$ and $\Hom_{\der^{b}(\ch)}(T,\Sigma^{i}M)=0$ for all $i\neq 0$.
 Hence $\dimv M$ is a positive $c$-vector of $A$ by Proposition~\ref{p:criterion}. In particular, we have proved that $\opname{\tau dv}(A)\subseteq \opname{cv}^+(A)$.
\end{proof}

\subsubsection{Piecewise hereditary algebras}
Recall that $\mathcal{H}$ is a hereditary abelian category with finite-dimensional morphism and extension spaces, $\der^{b}(\mathcal{H})$ is the bounded derived category of $\mathcal{H}$.
An object $T\in \der^{b}(\mathcal{H})$ is a {\it tilting complex} if
\begin{itemize}
\item[(1)] $\Hom_{\der^{b}(\ch)}(T,\Sigma^{n}T)=0$ for all $0\neq n\in \Z$;
\item[(2)] for each $X\in \der^{b}(\ch)$, the condition $\Hom_{\der^{b}(\ch)}(T, \Sigma^{n}X)=0$ for all $n\in \Z$ implies that $X=0$ in $\der^{b}(\ch)$.
\end{itemize}
A tilting complex  $T$ of $\der^{b}(\ch)$ is called a {\it tilting object} of $\mathcal{H}$ if $T\in \mathcal{H}$.  A finite-dimensional $k$-algebra $A$ is called {\it piecewise hereditary} if $A$ is isomorphic to the endomorphism algebra of a tilting complex $T$ in $\der^{b}(\ch)$. It is called {\it quasitilted} if moreover $T$ is a tilting object in $\mathcal{H}$.
A tilting complex $T$ induces an equivalence of triangulated categories $?\lten_{\End_{\der^{b}(\ch)}(T)}T:\der^{b}(\mod \End_{\der^{b}(\ch)}(T))\to \der^{b}(\mathcal{H})$.
By Happel's theorem~\cite{H01},  if  $\mathcal{H}$ is a connected hereditary abelian $k$-category with finite-dimensional morphism and extension spaces which  admits a tilting complex, then $\mathcal{H} $ is derived equivalent to the category $\mod H$ for certain finite-dimensional hereditary $k$-algebra or to the category $\opname{coh}\mathbb{X}$ of coherent sheaves over a weighted projective line~\cite{GL87}.
Note that when $\mathcal{H}$ is the category $\mod H$ for some finite-dimensional hereditary $k$-algebra,  the endomorphism algebra of a tilting module in $\mod H$ is called a {\it tilted algebra}~\cite{B81,HR82}. In particular, tilted algebras are quasitilted.

Let $A$ be a finite-dimensional $k$-algebra which is derived equivalent to $\mathcal{H}$. The algebra $A$ turns out to be piecewise hereditary. Indeed, let $K:\der^{b}(\mod A)\to \der^{b}(\ch)$ be the triangle equivalent functor. It is clear that $K(A)$ is a tilting complex of $\der^{b}(\ch)$ and $A\cong \End_{\der^{b}(\ch)}(K(A))$.
For a finite-dimensional $k$-algebra $A$,  a module $M$ is called {\it exceptional} provided  $\dim_{k}\Hom_{A}(M,M)=1$ and $\Ext^{1}_{A}(M,M)=0$.
We define
\[\opname{exdv}(A)=\{\dimv M| M\in \mod A ~\text{is exceptional}\}.\]
Our next result gives a upper bounded of positive $c$-vectors by exceptional modules for piecewise hereditary algebras.
\begin{proposition}~\label{t:upperbound}
Let $A$  be a piecewise hereditary $k$-algebra. Then we have \[\opname{\tau dv}(A)\subseteq \opname{cv}^+(A)\subseteq \opname{exdv}(A).\]
\end{proposition}
\begin{proof}

The first inclusion $\opname{\tau dv}(A)\subseteq \opname{cv}^{+}(A)$ follows from Proposition~\ref{p:lowerbound} directly.

Let $c$ be a positive $c$-vector,  by Proposition~\ref{p:criterion}, there is a $2$-term silting object $T\in \per A$ and an indecomposable $A$-module $M$ with $\dimv M=c$ such that $\Hom_A(T,M)=k$ and $\Hom_{A}(T,\Sigma^{i}M)=0$ for all $i\neq 0$. Let $\Gamma$ be the negative dg algebra associated to $T$ and $F:=\lten_{\Gamma}T:\der_{fd}(\Gamma)\to \der^b(\mod A)$ the triangle equivalent functor. The condition $\Hom_A(T,M)=k$ and $\Hom_{A}(T,\Sigma^{i}M)=0$ for all $i\neq 0$ implies  that there is a simple dg $\Gamma$-module $S$ such that $M\cong F(S)$. Moreover, the simple dg $\Gamma$-module $S$ is a simple object in the heart $\ca$ of the standard bounded $t$-structure $(\der^{\leq 0}, \der^{\geq 0})$ on $\der_{fd}(\Gamma)$.
 Let $\ch$ be the hereditary abelian category  such that $\der^{b}(\mod A)\cong \der^{b}(\ch)$, then we have $\der_{fd}(\Gamma)\cong \der^{b}(\ch)$. Therefore $\ca$ is equivalent to the heart of a bounded $t$-structure on $\der^{b}(\ch)$.
 On the other hand, we have $\ca\cong \mod \End_{A}(T)$. Hence $\Hom_{A}(M,M)=\Hom_{\Gamma}(S,S)=\Hom_{\ca}(S,S)=k$. Moreover,
by Corollary~\ref{c:nocycle}, we deduce that $0=\Hom_{\ca}(S,\Sigma S)=\Hom_{\Gamma}(S, \Sigma S)=\Ext_A^1(M,M)$. In particular, $c\in \opname{exdv}(A)$. This completes the proof.
\end{proof}

If $A$ is a finite-dimensional hereditary algebra over $k$, then we clearly have $\opname{\tau dv}(A)=\opname{exdv}(A)$. Hence we obtain the equalities for acyclic cluster algebras established by N\'{a}jera Ch\'{a}vez in ~\cite{Chavez13}.
\begin{corollary}
Let $A$ be a finite-dimensional hereditary algebra over $k$.  We have $\opname{cv}^{+}(A)=\opname{exdv}(A)$.
\end{corollary}

\subsubsection{Quasitilted algebras}
Let $A$ be a quasitilted algebra. By definition, there is a hereditary abelian $k$-category $\mathcal{H}$ with a tilting object $T\in \mathcal{H}$ such that $A\cong \End_{\mathcal{H}}(T)$. In this case, the category $\mod A$ of finitely generated right $A$-modules has a nice interpretation via torsion theory of $\mathcal{H}$.

Let $\ct$ (resp. $\cf$) be the full subcategory of $\mathcal{H}$ consisting of all objects $X$ (resp. $Y$) of $\mathcal{H}$ satisfying $\Ext_{\mathcal{H}}^{1}(T,X)=0$ (resp. $\Hom_{\mathcal{H}}(T,Y)=0$). Let $F:\der^{b}(\mod A)\to \der^{b}(\mathcal{H})$ be the triangle equivalent functor. Then the standard $t$-structure $(\der^{\leq 0}, \der^{\geq 0})$ on $\der^{b}(\mod A)$ induces a $t$-structure $(\der^{\leq0}_{T}, \der^{\geq 0}_{T})$ via the functor $F$ ({\it cf.} Section 2.4). Moreover, $\mod A$ is equivalent to the heart $\ca$ of $(\der^{\leq0}_{T}, \der^{\geq 0}_{T})$  via the functor $F$.  The following results are well-known, see ~\cite{B81,HR82,Lenzing07}.
\begin{lemma}~\label{l:projective-injective}
\begin{itemize}
\item[(1)]  $(\ct,\cf)$ is a torsion pair over $\mathcal{H}$;
\item[(2)]  $(\Sigma \cf, \ct)$ is a splitting torsion pair over $\ca$;
\item[(3)] Under the identification of $\mod A$ with $\ca$, we have $\opname{pd} M_{A}\leq 1$ for $M\in \ct$ and $\opname{id} N_{A}\leq 1$ for $N\in \Sigma \cf$.
\end{itemize}
\end{lemma}

\begin{theorem}
Let $A$ be a quasitilted algebra over $k$, then we have $\opname{cv}^{+}(A)=\opname{exdv}(A)$.
\end{theorem}
\begin{proof}
We need to show that  for each exceptional $A$-module $M$, the dimension vector $\dimv M$ is a positive $c$-vector of $A$.
We identify $\mod A$ with $\ca$ as above.  Let $M$ be an exceptional $A$-module. Since $(\Sigma \cf, \ct)$ is splitting, $M$ lies either in $\ct$ or in $\Sigma \cf$.
 If $M\in \ct$, then by Lemma~\ref{l:projective-injective} we have $\opname{pd} M\leq 1$.  As a consequence,  $M$ is an indecomposable $\tau$-rigid $A$-module. By Proposition~\ref{p:lowerbound}, we deduce that $\dimv M$ is a positive $c$-vector of $A$.

 Now suppose that $M\in \Sigma \cf$, then $\opname{id} M\leq 1$.  Recall that we have the usual duality $D=\Hom_{k}(?,k):\mod A\to \mod A^{\opname{op}}$. It is clear that $D(M)\in \mod A^{\opname{op}}$  is an exceptional $A^{\opname{op}}$-module  and has projective dimension at most one. Therefore $D(M)$ is an indecomposable $\tau$-rigid $A^{\opname{op}}$-module. On the other hand, we clearly have $\der^{b}(\mod A^{\opname{op}})\cong \der^{b}(\mathcal{H}^{{\opname{op}}})$, where $\mathcal{H}^{{\opname{op}}}$ is the opposite category of $\mathcal{H}$, which is a hereditary abelian category. By Proposition~\ref{p:lowerbound} again, we deduce that $\dimv D(M)$ is a positive $c$-vector of $A^{\opname{op}}$. Note that $\dimv M=\dimv D(M)$ and $\opname{cv}^{+}(A)=\opname{cv}^{+}(A^{\opname{op}})$, we have $\dimv M\in \opname{cv}^{+}(A)$. Now the result follows from  Proposition~\ref{t:upperbound}.
\end{proof}
\begin{remark}
Let $A$ be a finite dimensional algebra of finite global dimension. One can construct several Lie algebras related to $A$({\it cf.}~\cite{FuPeng14}). Namely, the Ringel-Hall Lie algebra $\mathfrak{CLRH}(A)$ in the sense of Peng-Xiao~\cite{PengXiao00}, the Borcherds type Lie algebra $\mathfrak{BL}(A)$ associated to the symmetric bilinear form of $A$~\cite{FuPeng14} and the intersection matrix Lie algebra $im(A)$ in the sense of Slodowy~\cite{Slodowy86}. All of these Lie algebras are graded by the Grothendieck group $\go(\mod A)$ of $A$. When $A$ is a tilted algebra, all of the three Lie algebras are isomorphic to a Kac-Moody Lie algebra $\mathfrak{g}(A)$. Let $\phi$ be the induced isomorphism between $\go(\mod A)$ and the root lattice of $\mathfrak{g}(A)$. We call a root $\alpha\in \go(\mod A)$ of $im(A)$ is a {\it real Schur root} of $A$ if the image $\phi(\alpha)$ is a real Schur root of $\mathfrak{g}(A)$. Then the above theorem implies that for a tilted algebra $A$, its positive $c$-vectors coincide with the positive real Schur roots of $A$
\end{remark}

\subsection{$c$-vectors of representation-directed algebras}~\label{S:c-vectors of rep-directed}
A finite-dimensional  $k$-algebra  is called {\it representation-directed} if there is no cycle in $\mod A$.
Let $A$ be a finite-dimensional  representation-directed $k$-algebra. By the definition, we know that every indecomposable $A$-module is $\tau$-rigid and also exceptional.
We have the following equality.
\begin{proposition}
Let $A$ be a representation-directed algebra over  $k$. We have $\opname{cv}^+(A)=\opname{exdv}(A)$.
\end{proposition}
\begin{proof}
We need to  prove the dimension vector of each indecomposable $A$-module is a positive $c$-vector.
Let $M$ be an arbitrary indecomposable $A$-modules. Consider the torsion class $\opname{Fac}M$ generated by $M$, which is a  functorially finite torsion class.  Let  $N:=P(\opname{Fac}M)$ be the direct sum of one copy of each of the indecomposable $\Ext$-projective objects in $\opname{Fac}M$. We clearly have $M\in \add N$. By Theorem~\ref{t:bijection fftorsion-tautilting}, we deduce that $N$ is a support $\tau$-tilting $A$-module. Equivalently, there is a projective $A$-module $P$ such that $(N, P)$ is a support $\tau$-tilting pair. Let $T$ be the corresponding $2$-term silting complex.  A direct calculation shows that $\Hom_A(T, M)\cong k$ and $\Hom_A(T, \Sigma^iM)=0$ for $i\neq 0$.
  In particular, the dimension vector $\dimv M$ is a $c$-vector of $A$ by Proposition~\ref{p:criterion}. This finishes the proof.
\end{proof}
Note that an algebra derived equivalent to a representation-finite hereditary algebra has to be a representation-direct algebra. We have the following special case of the above result.
\begin{corollary}
Let $A$ be finite-dimensional algebra over $k$. If  $A$ is derived equivalent to a representation-finite hereditary algebra, then $\opname{cv}^+(A)=\opname{exdv}(A)$.
\end{corollary}

\subsection{$c$-vectors of cluster-tilted algebras}~\label{s:cluster-tilted}
We follow ~\cite{BMRRT}. Let $H$ be a finite-dimensional hereditary $k$-algebra and $\der^b(H)$ the bounded derived category of finitely generated right $H$-modules. Denote by $\Sigma$ the suspension functor of $\der^b(H)$ and $\tau$ the Auslander-Reiten translation functor.  The cluster category $\cc_H$ has been introduced in ~\cite{BMRRT} as the orbit category $\der^b(H)/ \tau^{-1}\circ \Sigma$ of $\der^b(H)$. It admits a canonical triangle structure such that the projection $\pi_H:\der^b(H)\to \der^b(H)/\tau^{-1}\circ \Sigma$ is a triangle functor~\cite{Keller05}.  An object $T$ in $\cc_H$ is called a {\it cluster-tilting object} provided that
\begin{itemize}
\item[$\circ$]$\Hom_{\cc_H}(T, \Sigma T)=0$;
\item[$\circ$] if $X\in \cc_H$ such that $\Hom_{\cc_H}(T,\Sigma X)=0$, then $X\in \add T$.
\end{itemize}
Let $n$ be the number of pairwise non-isomorphic simple $H$-modules. Each basic cluster-tilting object in $\cc_H$ has exactly $n$ indecomposable direct summands.
The endomorphism algebra $\End_{\cc_H}(T)$ of a basic cluster-tilting object $T\in \cc_H$ is a {\it cluster-tilted algebra} of type $H$ ({\it cf.}~\cite{BMR07}). It is known that cluster-tilted algebras are 1-Gorenstein algebras. Moreover, the functor $\Hom_{\cc_H}(T,?):\cc_H\to \mod \End_{\cc_H}(T)$ yields an equivalence $\cc_H/\Sigma T\cong \mod \End_{\cc_H}(T)$, where $\cc_H/\Sigma T$ is the additive quotient of $\cc_H$ by the morphism factorizing through $\Sigma T$ ({\it cf.}~\cite{BMR07, KR07}).

\begin{proposition}
Let $H$ be a finite-dimensional hereditary $k$-algebra and $\cc_H$ the corresponding cluster category. Let $T$ be a cluster-tilting object and $A$ the endomorphism algebra of $T$. Let $M$ be an indecomposable preprojective or preinjective $H$-module such that $\Hom_{\cc_H}(T,M)\neq 0$, then the dimension vector $\dimv \Hom_{\cc_H}(T,M)$ of $A$-module is a positive $c$-vector of $A$.
\end{proposition}
\begin{proof}
It is easy to see that there is a cluster tilting object $T_M=M\oplus M_1\cdots\oplus M_{n-1}$ such that $\Hom_{\cc_H}(M_i,M)=0$ for all $1\leq i\leq n-1$. Applying the functor $\Hom_{\cc_H}(T,?)$, we deduce that $N_A:=\Hom_{\cc_H}(T,T_M)$ is a support $\tau$-tilting $A$-module. Let $P$ be the $2$-term silting object corresponding to $N_A$. A direct calculation shows that \[\Hom_A(P,\Sigma^i \Hom_{\cc_H}(T,M))=\begin{cases}k&i=0;\\0&\text{else}.\end{cases}\]
In particular, $\dimv \Hom_{\cc_H}(T,M)$ is a positive $c$-vector of $A$ by Proposition~\ref{p:criterion}.
\end{proof}
Note that for a representation-finite hereditary algebra $H$, each $H$-module is a preprojective module.
As a consequence, we recover the following equality of $c$-vectors  for skew-symmetric cluster algebras of finite type in~\cite{Chavez14}.
\begin{corollary}
Let $A$ be a cluster-tilted algebra of representation-finite type. We have $\opname{cv}^+(A)=\opname{dv}(A)$.
\end{corollary}


\begin{appendix}
\section{Sign-coherence of c-vectors for skew-symmetric cluster algebras}~\label{S:appendix}
In this section, we recall the tropical dualities between $c$-vectors and $g$-vectors for cluster algebras with principal coefficients.
 We give an interpretation of cluster algebras of finite type via the finiteness of $c$-vectors and a short proof for the sign-coherence property of $c$-vectors for skew-symmetric cluster algebras.

\subsection{ Cluster algebras with principal coefficients and $c$-vectors}
We follow ~\cite{FZ07,FuKeller10}.
For an integer $x$, we set $[x]_+=\opname{max}\{x,0\}$ and $\opname{sgn}(x)=\begin{cases}\frac{x}{|x|}&x\neq 0\\0&x=0\end{cases}$.
Let $1\leq n\leq m\in \N$. Let $\mathbb{QP}$ be the algebra of Laurent polynomials in the variables $x_{n+1}, \cdots, x_{m}$ and $\mathcal{F}$ the field of fractions of
the ring of polynomials with coefficients in $\mathbb{QP}$ in $n$ indeterminates.
A matrix $B\in M_n(\mathbb{Z})$ is {\it skew-symmetrizable} if there exists a diagonal martrix $D=\operatorname{diag}\{d_1,\cdots, d_n\}$ with positive integer entries  such that $DB$ is skew-symmetric. In this case, the matrix $D$ is call a {\it skew-symmetrizer} of $B$.
A {\it seed} in $\mathcal{F}$ is a pair $(\wt{B}, \mathbf{x})$ consisting of an $m\times n$ integer matrix $\wt{B}$ whose {\it principal part} (that is the submatrix formed by the first $n$ rows) is skew-symmetrizable  and a free generating set $\mathbf{x}=\{x_1,x_2, \dots, x_n\}$ of the field $\mathcal{F}$. The matrix $\wt{B}$ is  the {\it exchange matrix } and $\mathbf{x}$ is the {\it cluster} of the seed $(\wt{B},\mathbf{x})$. Elements of the cluster $\mathbf{x}$   are  {\it cluster variables} of the seed $(\wt{B}, \mathbf{x})$.

 For any $1\leq k\leq n$, the {\it seed mutation of $(\widetilde{B},\mathbf{x})$ in the direction $k$} transforms $(\widetilde{B},\mathbf{x})$ into a new seed $\mu_k(\widetilde{B}, \mathbf{x})=(\widetilde{B}',\mathbf{x}')$, where
 \begin{itemize}
 \item[$\bullet$] the entries $b_{ij}'$ of $\wt{B}'$ are given by
\[b_{ij}'=\begin{cases}-b_{ij} &\text{if }i=k ~\text{or}~j=k,\\
b_{ij}+\opname{sgn}(b_{ik})[b_{ik}b_{kj}]_{+}& \text{else}.\end{cases}
\]
\item[$\bullet$] the cluster $\mathbf{x}'=\{x_1',\cdots, x_n'\}$ is given by $x_j'=x_j$ for $j\neq k$ and $x_k'\in \mathcal{F}$ is determined by the {\it exchange relation}
\[x_k'x_k=\prod_{i=1}^mx_i^{[b_{ik}]_+}+\prod_{i=1}^mx_i^{[-b_{ik}]_+}.
\]
\end{itemize}
Mutation in a fixed direction is an involution.  The {\it cluster algebra with coefficients} $\ca(\wt{B})=\ca(\wt{B}, \mathbf{x})$ is the subalgebra of $\mathcal{F}$ generated by all the cluster variables which can be obtained from the initial seed $(\wt{B}, \mathbf{x})$ by iterated mutations.
We call a cluster algebra {\it skew-symmetric type} if the principal part of its initial exchange matrix is skew-symmetric.
 If $m=2n$ and the {\it coefficient part} of initial exchange matrix $\wt{B}$ (that is the submatrix formed by the last $m-n$ rows) is the identity matrix $E_n$, then $\ca(\wt{B})$ is a {\it  cluster algebra with principal coefficients}.

 Let $\mathbb{T}_n$ be the $n$-regular tree, whose edges are labeled by the numbers $1,2,\cdots, n$ so that the $n$ edges emanating from each vertex carry different labels.
 A {\it cluster pattern} is the assignment of a seed $(\wt{B}_t,\mathbf{x}_t)$ to each vertex $t$ of $\mathbb{T}_n$ such that the seeds assigned to vertices $t$ and $t'$ linked by an edge labeled $k$ are obtained from each other by  the seed mutation $\mu_k$. A cluster pattern is uniquely determined by an  assignment of the initial seed $(\wt{B}, \mathbf{x})$ to any vertex $t_0\in \mathbb{T}_n$.

 Let $\ca(\widetilde{B})$ be a cluster algebra with principal coefficients. We fix a cluster pattern of $\ca(\wt{B})$.
For any $t\in \mathbb{T}_n$, let $C_t$ be the coefficient part of the matrix $\widetilde{B}_t$. Each column vector of $C_t$ is called a {\it $c$-vector} of $\ca(\widetilde{B})$.
The following sign-coherence property has been conjectured by Fomin-Zelevinsky~\cite{FZ07} for any cluster algebra $\ca(\widetilde{B})$ with principal coefficients, which has been established in~\cite{GHKK14} recently.
\begin{theorem}~\label{t:sign-coherenceofclusteralgebra}
Each $c$-vector of $\ca(\widetilde{B})$ is sign-coherence.
\end{theorem}
Let us mention here  that for skew-symmetric cluster algebras, the sign-coherence conjecture has been confirmed by Derksen-Weyman-Zelevinsky~\cite{DWZ10}, Plamondon~\cite{Plamondon11} and Nagao~\cite{Nagao13}.
Demonet~\cite{Demonet10} also established the sign-coherence conjecture for  certain skew-symmetrizable cluster algebras such as cluster algebras admit unfoldings. In the end of this section,
We shall also give a direct proof for the sign-coherence property of skew-symmetric cluster algebras basing on Proposition $6.10$ of ~\cite{FuKeller10}.
 \subsection{$g$-vectors and tropical dualities}
Let $\ca(\widetilde{B})$ be a cluster algebra with principal coefficients. We fix a cluster pattern of $\ca(\widetilde{B})$ as before.

Let $x_1, \cdots, x_n$ be the initial cluster variables.
A fundamental result of Fomin-Zelevinsky~\cite{FZ02} says that each cluster variable $x_j^t$ of the cluster $\mathbf{x}_t$ is expressed as \[X_j^t(x_1, \cdots,x_{2n})\in \Z[x_1^{\pm}, \cdots, x_n^{\pm}, x_{n+1}, \cdots, x_{2n}].\]  Set $\deg (x_i)=e_i $ and $\deg (x_{n+i})=-b_i$ for $1\leq i\leq n$, where $e_1, \cdots, e_n$ is the standard basis of $\Z^n$ and $b_i$ is the $i$-th column of the principal part of $\widetilde{B}$. This makes $\Z[x_1^{\pm}, \cdots, x_n^{\pm}, x_{n+1}, \cdots, x_{2n}]$ into a $\Z^n$-graded ring. In ~\cite{FZ07}, they proved that each $X_j^t$ is homogeneous with respect to the $\Z^n$-grading. Denote by $\deg (X_j^t)=(g_{1j}, \cdots, g_{nj})'\in \Z^n$ the grading of $X_j^t$ and set $G_t=(g_{ij})_{i,j=1}^n$. The matrix $G_t$ is called the {\it $G$-matrix} of $\ca(\widetilde{B})$ at vertex $t$ and its column vectors are {\it $g$-vectors}. The $g$-vectors were introduced in ~\cite{FZ07} to  parameterize cluster variables and they conjectured that
\begin{conjecture}~\label{con:g-vector}
Different cluster variables have different $g$-vectors.
\end{conjecture}
This conjecture has been verified for skew-symmetric cluster algebras by Derksen-Weyman-Zelevinsky~\cite{DWZ10}, Plamondon~\cite{Plamondon11}, Nagao~\cite{Nagao13}  by using representation theory of quivers with potentials and for certain skew-symmetrizable cluster algebras ({\it e.g} cluster algebras admit unfoldings) by Demonet~\cite{Demonet10}.
The definition of $g$-vectors is quite different from the one of $c$-vectors, but a recently work of Nakanishi~\cite{N11}  and Nakanishi-Zelevinsky~\cite{NZ12} showed that  there are so-called tropical dualities between $c$-vectors and $g$-vectors. The following theorem has been proved in~\cite{N11} for skew-symmetric cluster algebras and then generalized to skew-symmetrizable cluster algebras in ~\cite{NZ12} by assuming the sign-coherence conjecture.
\begin{theorem}~\label{t:tropicalduality}
 Let $\ca(\widetilde{B})$ be a skew-symmetrizable cluster algebra with principal coefficients whose skew-symmetrizer is $D$, then for each vertex $t\in \mathbb{T}_n$,
\[G_t'DC_t=D.
\]
\end{theorem}

\subsection{Cluster algebras of finite type via $c$-vectors}
A cluster algebra $\ca(\widetilde{B})$ is called of {\it finite type} if there are only finitely many different cluster variables. Cluster algebras of finite type are classified by Fomin-Zelevinsky in ~\cite{FZ03}. A much broader class of cluster algebras are  of mutation finite type. A cluster algebra $\ca(\widetilde{B})$ is of {\it mutation finite type} if the mutation-equivalent class of the initial matrix $\widetilde{B}$ is finite. In this case, we also call the matrix $\widetilde{B}$ is of {\it mutation finite type}. It is known that cluster algebras of finite types are of mutation finite types, but the inverse is not true. However, for cluster algebras with principal coefficients, Fomin-Zelevinsky~\cite{FZ07} conjectured that
\begin{conjecture}
Let $\ca(\widetilde{B})$ be a cluster algebra with principal coefficients, then $\ca(\widetilde{B})$ is of finite type if and only if the initial matrix $\widetilde{B}$ is of mutation finite type.
\end{conjecture}
Recently, this conjecture has been  proved by Seven~\cite{Seven13} basing on the classification of cluster algebras of {\it minimal infinite type}~\cite{Seven11}. In fact, his proof implies the following statement.
\begin{theorem}
Let $\ca(\widetilde{B})$ be a cluster algebra with principal coefficients, then $\ca(\widetilde{B})$ is of finite type if and only if  the set of $c$-vectors $\opname{cv}(\ca(\widetilde{B}))$ of $\ca(\widetilde{B})$ is finite.
\end{theorem}
In the following, we sketch  a proof basing on tropical dualities between $c$-vectors and $g$-vectors.
Note that only the if part needs a proof.
The above result follows easily from the Theorem~\ref{t:sign-coherenceofclusteralgebra}  and Conjecture ~\ref{con:g-vector}. Namely, if the set $\opname{cv}(\ca{\widetilde{B}})$ is finite, then there are only finitely many different $t\in \mathbb{T}_n$ such that the $C$-matrices $C_t$ are pairwise different. Now by Theorem ~\ref{t:tropicalduality}, we deduce that there are finitely many $t\in \mathbb{T}_n$ such that the $G$-matrices $G_t$ are pairwise different. Hence there are finitely many different cluster variables by Conjecture~\ref{con:g-vector} and  $\ca(\widetilde{B})$ is of finite type.

 Note that  Conjecture~\ref{con:g-vector} are confirmed for all the skew-symmetric cluster algebras and for certain skew-symmetrizable cluster algebras such as cluster algebras admit unfoldings. In particular, the above proof can be applied to  skew-symmetric cluster algebras with principal coefficients.

Now, let us assume that $\ca(\widetilde{B})$ is skew-symmetrizable. Let $B$ be the principal part of $\widetilde{B}$. If $\opname{cv}(\ca(\widetilde{B}))$ is finite, then one can show that  $B$ is of mutation finite type by using Proposition 2.7 of ~\cite{Seven13}. It follows from the classification of cluster algebras of  mutation finite type~\cite{FST12} that $\ca(\widetilde{B})$ admits an unfolding. By Demonte's result~\cite{Demonet10}, we know that Conjecture~\ref{con:g-vector} is true in this case and hence that above proof applies.

\subsection{Quivers with potentials and mutations}
We follow ~\cite{DWZ08,Keller12}.
Let $Q=(Q_0,Q_1)$ be a finite quiver, where $Q_0$ is the set of vertices and $Q_1$ is the set of arrows. Let $kQ$ be the path algebra of $Q$ and $\hat{kQ}$ its completion with respect to path length. Thus, $\hat{kQ}$ is a topological  algebra and the paths of $Q$ form a topological basis. The {\it continuous zeroth Hochschild homology} of $\hat{kQ}$ is the vector space $HH_0(\hat{kQ})$ obtained as the quotient of $\hat{kQ}$ by the closure of the subspace $[\hat{kQ}, \hat{kQ}]$ of  all commutators.
It has a topological basis formed by the classes of cyclic paths of $Q$. For each arrow $a$  of $Q$, the {\it cyclic derivative} with respect to $a$ is the unique continuous map
\[\partial_a: HH_0(\hat{kQ})\to \hat{kQ}
\]
which takes the class of a path $p$ to the sum
\[\sum_{p=uav}vu
\]
taken over all decompositions of $p$ as a concatenation of path $u, a,v $, where $u,v$ are of length $\geq 0$. A {\it potential} on $Q$ is an element $W$ of $HH_0(\hat{kQ})$ which does not involve cycles of length $\leq 2$.

Let $(Q,W)$ be a quiver with  potential  and $k$ a vertex of $Q$.  With certain mild condition for $(Q,W)$ at the vertex $k$, Derksen-Weyman-Zelevinsky~\cite{DWZ08}  introduced the {\it mutation of $(Q,W)$ in the direction $k$ } which transforms $(Q,W)$ into a new quiver with potential $\mu_k(Q,W)=(Q',W')$ (for the precisely definition, we refer to ~\cite{DWZ08}).
In this case, we call the quiver with potential $(Q,W)$ is {\it mutable at vertex $k$}.
The quiver $Q$ and $Q'$ have the same vertices but different arrows.
In general, the resulting quiver with potential $\mu_k(Q,W)$ may not be mutable at certain vertices. But if the quiver $Q$ has no loops nor $2$-cycles, there exists a {\it non-degenerate} potential $W$ on $Q$ such that we can indefinitely mutate the quiver with potential $(Q,W)$. Moreover, each quiver with potential obtained from $(Q,W)$ by iterated mutations  has no loops nor $2$-cycles.

Let $Q$ be a finite quiver without loops nor $2$-cycles with vertex set $Q_0=\{1,2,\cdots, m\}$. We define an $m\times m$ integer matrix $B(Q)$ associated to $Q$ such that
\[b_{ij}=|\{\text{arrows from vertex $i$ to vertex $j$}\}|-|\{\text{arrows from vertex $j$ to vertex $i$}\}|.
\]
Conversely, for a given integer skew-symmetric matrix $B$, there is a unique quiver $Q$ without loops nor 2-cycles such that $B(Q)=B$. Let $W$ be a non-degenerate potential on $Q$. We may assign each vertex $t\in \mathbb{T}_m$ a quiver with potential $(Q_t,W_t)$ which can be obtained from $(Q,W)$ by iterated mutations such that the quivers with potentials assigned to $t$ and $t'$ linked by an edge labeled $k$ are obtained from each other by the mutation  $\mu_k$. By Proposition $7.1$ in \cite{DWZ08},  if $(Q_t,W_t)$ and $(Q_{t'},W_{t'})$ are linked by an edge $k$, then we have  $B(Q_t)=\mu_k(B(Q_{t'}))$.

\subsection{Ginzburg dg algebras and derived equivalences}
Let $Q$ be a finite quiver and $W$ a potential on $Q$. The Ginzburg dg algebra $\Gamma_{(Q,W)}$ of $(Q,W)$ introduced by Ginzburg~\cite{Ginzburg06} is constructed as follows: Let $\overline{Q}$ be the graded quiver with the same vertices as $Q$ and whose arrows are
\begin{itemize}
\item[$\circ$] the arrow of $Q$, which are of degree $0$;
\item[$\circ$] an arrow $a^*:j\to i$ of degree $-1$ for each arrow $a:i\to j$ of $Q$;
\item[$\circ$] a loop $t_i:i\to i$ of degree $-2$ for each vertex $i$ of $Q$.
\end{itemize}
The underlying graded algebra of $\Gamma_{(Q,W)}$ is the completion of the graded path algebra $k\overline{Q}$ in the category of graded vector spaces with respect to the ideal generated by the arrows of $\overline{Q}$. In particular, the $n$-component of $\Gamma_{(Q,W)}$ consisting of elements of the form $\sum_{p}\lambda_p p$, where $p$ runs over all paths of degree $n$. The differential $d$ of $\Gamma_{(Q,W)}$ is the unique continuous linear endomorphism homogenous of degree $1$ which satisfies the Leibniz rule
\[d(uv)=(du)v+(-1)^pudv,
\]
for all homogeneous $u$ of degree $p$ and all $v$, and takes the following values on the arrows of $\overline{Q}$:
\begin{itemize}
\item[$\circ$] $da=0$ for each arrow $a$ of $Q$;
\item[$\circ$] $d(a^*)=\partial_aW$ for each arrow $a$ of $Q$;
\item[$\circ$] $d(t_i)=e_i(\sum_a[a,a^*])e_i$ for each vertex $i$ of $Q$, where $e_i$ is the lazy path at $i$ and the sum runs over the set of arrows of $Q$.
\end{itemize}

Let $Q$ be a finite quiver without loops nor $2$-cycles with vertex set $\{1,2,\cdots, m\}$ and $W$ a non-degenerate potential on $Q$. Denote by $\Gamma_{(Q,W)}$  the Ginzburg dg algebra associated to $(Q,W)$. Let $k$ be a vertex of $Q$ and $\Gamma_{\mu_k(Q,W)}$ the Ginzburg dg algebra associated to $\mu_k(Q,W)$.  Let $e_1,\cdots, e_m$ be the idempotents of $\Gamma_{(Q,W)}$ and $\Gamma_{\mu_k(Q,W)}$ associated to the vertices of $(Q,W)$ and $\mu_k(Q,W)$. Let $\der(\Gamma_{(Q,W)})$ and $\der(\Gamma_{\mu_k(Q,W)})$ be the derived categories of $\Gamma_{(Q,W)}$ and $\Gamma_{\mu_k(Q,W)}$ respectively.
The  following result is due to Keller-Yang~\cite{KellerYang11}.
\begin{theorem}~\label{t:Keller-Yang}
There is a triangle equivalence
\[\Phi:\der(\Gamma_{\mu_k(Q,W)})\to \der(\Gamma_{(Q,W)})
\]
which sends  the $e_i\Gamma_{\mu_k(Q,W)}$ to $e_i\Gamma_{(Q,W)}$ for $i\neq k$ and to the   mapping cone of the morphism $e_k\Gamma_{\mu_k(Q,W)}\to \oplus_{k\to j}e_j\Gamma_{(Q,W)}$  for $i=k$, where the sum is taken over the arrows in $Q$.
\end{theorem}

\subsection{A short proof of sign-coherence conjecture for skew-symmetric cluster algebras}
Let $\ca(\wt{B})$ be a skew-symmetric cluster algebra with principal coefficients.
We fix a cluster pattern of $\ca(\wt{B})$ by assigning the initial seed $(\wt{B},\mathbf{x})$ to the vertex $t_0\in \mathbb{T}_n$.

Let $Q=(Q_0,Q_1)$ be a finite quiver without loops nor $2$-cycles with vertex set $Q_0=\{1,2,\cdots,n\}$ such that $B(Q)$ is the principal part of the initial matrix $\widetilde{B}$.
We define a new quiver $\widetilde{Q}$ such that the set of vertices $\widetilde{Q}_0=Q_0\cup \{1+n,2+n, \cdots, 2n\}$ and the set of arrows $\widetilde{Q}_1=Q_1\cup\{i+n\to i|i\in Q_0\}$.
Let $W$ be a non-degenerate potential on $\widetilde{Q}$.
We may assign each vertex $t\in \mathbb{T}_n$ a quiver with potential $(\widetilde{Q}_t, W_t)$ which can be obtained from $(\widetilde{Q},W)$ by iterated mutations of $\mu_k$ for $1\leq k\leq n$ such that the quivers with potentials assigned to $t$ and $t'$ linked by an edge labeled $k$ are obtained from each other by the mutation  $\mu_k$.
Let $(\widetilde{Q}_{t_0},W_{t_0})$ be the quiver with potential $(\widetilde{Q},W)$. For each quiver with potential $(\widetilde{Q}_t,W_t)$, let $B(\widetilde{Q}_t)$ be the corresponding skew-symmetric matrix and $B(\widetilde{Q}_t)^\circ$  the submatrix of $B(\widetilde{Q}_t)$ formed by the first $n$ columns.
Recall that for each vertex $t\in \mathbb{T}_n$, we have a seed $(\wt{B}_t, \mathbf{x}_t)$ by the fixed cluster pattern.
By Proposition $7.1$ in ~\cite{DWZ08}, we know that $B(\widetilde{Q}_t)^\circ=\wt{B}_t$ for all $t\in \mathbb{T}_n$. Let $C_t=(c_{ij}^t)\in M_n(\Z)$ be the coefficient part of $\widetilde{B}_t$, we clearly have
\[c^t_{ij}=|\{\text{arrows from vertex $i+n$ to vertex $j$}\}|-|\{\text{arrows from vertex $j$ to vertex $i+n$}\}|.
\]

Note that we have $\Hom_{\der(\Gamma_{(Q,W)})}(e_{i+n}\Gamma_{(Q,W)}, e_{j+n}\Gamma_{(Q,W)})=0$ for any $1\leq i\neq j\leq n$. It follows from Theorem~\ref{t:Keller-Yang} that there is no arrow between vertex $i+n$ and  $j+n$ in the  quiver $\widetilde{Q}_t$ for any $t\in \mathbb{T}_n$.
Suppose that there is a vertex $t\in \mathbb{T}_n$ such that the $k$th column vector of $C_t$ is not sign-coherence. Hence there exist vertices $i+n$ and $j+n$ for $1\leq  i\neq j\leq n$ such that $c^t_{ik}>0$ and $c^t_{jk}<0$. Now consider the mutation at vertex $k$, we obtain that in the quiver with potential $\mu_k(\widetilde{Q}_t,W_t)$ there are $c^t_{ik}\times c^t_{jk}$ arrows from $i+n$ to $j+n$,  a contradiction.

\end{appendix}

\end{document}